\documentclass[12pt,reqno]{amsart} 

\usepackage{mathrsfs,amsfonts,amsthm,amsmath,amssymb,thmtools} 
\usepackage[dvipsnames]{xcolor} 

\usepackage{enumitem}                
\usepackage{lineno}                 

\usepackage[numbers,sort&compress]{natbib}  

\usepackage{graphics}
\usepackage{graphicx}   
\graphicspath{{images/}} 
\usepackage{subfig}     
\usepackage{rotating}

\usepackage[hyperindex, pdfencoding = auto, psdextra, bookmarksdepth = 4]{hyperref}  
\hypersetup{
    bookmarksopen      = true,
    bookmarksopenlevel = 1,
    bookmarksnumbered  = true,
    pdftitle    = {},
    pdfauthor   = {},
    linktoc     = page,
    anchorcolor = green,
    colorlinks  = true,        
    linkcolor   = black,
    citecolor   = blue,
    filecolor   = blue,
    urlcolor    = magenta,
}

\usepackage[capitalise, noabbrev, nameinlink]{cleveref}    

\newtheorem{innercustomthm}{Theorem}
\newenvironment{customthm}[1]
    {\renewcommand\theinnercustomthm{#1}\innercustomthm}
    {\endinnercustomthm}

\theoremstyle{plain}

\newtheorem{theorem}{Theorem}[section]
\newtheorem{lemma}[theorem]{Lemma}

\newtheorem{proposition}[theorem]{Proposition}
\newtheorem{corollary}[theorem]{Corollary}

\theoremstyle{definition}   
\newtheorem{definition}[theorem]{Definition}

\makeatletter
    \def\subsection{\@startsection{subsection}{2}%
    \z@{.5\linespacing\@plus.7\linespacing}{.3\linespacing}%
    {\normalfont\bfseries}}
\makeatother 

\makeatletter                                         
    \newcommand{\LeftEqNo}{\let\veqno\@@leqno}        
\makeatother

\numberwithin{equation}{section}                      

\begin{document}

\

\vspace{-2cm}

\title[The Ulam--Kaplansky problem in operator algebras]
{The Ulam--Kaplansky problem in operator algebras}

\author{Chunlan Jiang}
\address{Chunlan Jiang, School of Mathematics, Hebei Normal University, Shijiazhuang, 050016, China}
\email{cljiang@hebtu.edu.cn}

\author{Minghui Ma}
\address{Minghui Ma, School of Mathematical Sciences, Dalian University of Technology, Dalian, 116024, China}
\email{minghuima@dlut.edu.cn}

\author{Rui Shi}
\address{Rui Shi, School of Mathematical Sciences, Dalian University of Technology, Dalian, 116024, China}
\email{ruishi@dlut.edu.cn}

\author[Yuanhang Zhang]{Yuanhang Zhang}
\address{Yuanhang Zhang, School of Mathematics, Jilin University, Changchun, 130012, China}
\email{zhangyuanhang@jlu.edu.cn}


\makeatletter
\@namedef{subjclassname@2020}{\textup{2020} Mathematics Subject Classification}
\makeatother

\begin{abstract}
The Ulam--Kaplansky problem up to similarity in operator algebras asks whether $A\oplus A\sim B\oplus B$ implies $A\sim B$.
We construct block-diagonal compact operators for which this implication fails.
For von Neumann algebras with separable predual, we prove that this cancellation property is equivalent to subhomogeneity.
Moreover, we obtain a cancellation law for finite ampliations in unital Banach algebras when the commutant of one element is finite dimensional modulo its Jacobson radical.
\end{abstract}

\maketitle

\section{Introduction}

In 1935, Ulam asked in Problem~77 of \emph{The Scottish Book} \cite[p.~157]{Mauldin-15} whether two objects must be equivalent if their Cartesian squares are equivalent.
Part~$(c)$ of Problem~77 asks whether $G\times G\cong H\times H$ implies $G\cong H$ for arbitrary groups $G$ and $H$.
In the language of classification, the question is whether the map $[G]\mapsto[G\times G]$ is injective on isomorphism classes of groups.
Nearly two decades later, in his little red book \emph{Infinite Abelian Groups} \cite[p.~12]{Kaplansky-54}, Kaplansky proposed \emph{three test problems} for evaluating structure theorems for abelian groups.
The second problem asks whether $G\oplus G\cong H\oplus H$ implies $G\cong H$ for abelian groups $G$ and $H$.
As Kaplansky himself noted, all three test problems can be formulated for general mathematical systems, and he also provided commentary on Problem 77 (c) \cite[pp.~158--159]{Mauldin-15}.

In the context of Hilbert space operators, the \emph{Ulam--Kaplansky problem} is: if $T, S$ are operators in $\mathcal{B}(\mathcal{H})$ and $T\oplus T$ is equivalent to $S\oplus S$ in $\mathcal{B}(\mathcal{H}\oplus\mathcal{H})$, is $T$ equivalent to $S$?
The same problem can also be asked in operator algebras.
In essence, one is seeking an \emph{appropriate} multiplicity theory for the notion of equivalence involved.

To consider the Ulam--Kaplansky problem, one must first clarify the equivalence relation.
As suggested by Kaplansky, Kadison and Singer \cite{KS-57} applied this problem to unitary equivalence of operators and proved the interesting Ulam--Kaplansky result that $T\oplus T\cong S\oplus S$ implies $T\cong S$.
With the help of Voiculescu's noncommutative Weyl--von Neumann theorem \cite[Theorem II.5.8]{Davidson-96}, the Ulam--Kaplansky problem up to approximate unitary equivalence also admits a positive solution \cite[Proposition 4.4]{MRTZ-24}.
In 1990, Davidson and Herrero remarked that Kaplansky's test problems up to similarity of operators remained open and were quite interesting \cite[p.~53]{DH-90}.
In the similarity case, positive answers are established for several classes of operators.
In \cite[Theorem 4.6]{Jiang-04}, it was proved that if $T$ and $S$ are both strongly irreducible Cowen--Douglas operators, then $T$ and $S$ are similar (denoted by $T\sim S$) if and only if $T\oplus T$ and $S\oplus S$ are similar.
In \cite[Theorem 4.17]{MRTZ-24}, a positive solution to the Ulam--Kaplansky problem up to similarity has been obtained for the case where $T$ is an isometry operator.

In the current paper, the purpose is to systematically investigate the Ulam--Kaplansky problem up to similarity in von Neumann algebras and in general Banach algebras.
Our quantitative approach uses the similarity functions $\{\theta_n\}_{n=1}^\infty$ introduced in \Cref{def similarity}.
We prove in \Cref{prop upper,prop lower} that
\begin{equation*}
	\frac{1}{500}(t-1)\sqrt n\leqslant\theta_n(t)\leqslant(10t)^{24^n}\quad(n\geqslant 2,~t\geqslant 1).
\end{equation*}
The upper bound for $\theta_n(t)$, together with \Cref{prop reduction}, gives an affirmative answer to the Ulam--Kaplansky problem in type $\mathrm{I}_n$ von Neumann algebras.
The growth of the lower bound with $n$ gives compact counterexamples to similarity cancellation.

\begin{customthm}{\ref{thm UKP-1}}
Let $\mathcal H$ be an infinite-dimensional Hilbert space.
For every $t>1$, there exist block-diagonal compact operators $A,B\in\mathcal B(\mathcal H)$ and an invertible operator $X\in\mathcal B(\mathcal H\oplus\mathcal H)$ such that
\begin{equation*}
    A^{(2)}X=XB^{(2)},\qquad
    \|X\|\|X^{-1}\|\leqslant t,
    \qquad
    A\nsim B.
\end{equation*}
\end{customthm}

The following theorem characterizes the von Neumann algebras with separable predual for which the Ulam--Kaplansky problem holds.

\begin{customthm}{\ref{thm UKP-2}}
Let $\mathcal M$ be a von Neumann algebra with separable predual.
Then the Ulam--Kaplansky problem holds for $\mathcal M$ if and only if $\mathcal M$ is subhomogeneous.
\end{customthm}

By \Cref{thm UKP-2}, the Ulam--Kaplansky problem does not hold for general von Neumann algebras.
Thus, it is natural to ask under what conditions on operators the Ulam--Kaplansky problem holds.

We next give a sufficient condition for similarity cancellation in unital Banach algebras, formulated in terms of commutants.
We say that $T\in\mathcal B$ has an \emph{essentially finite-dimensional commutant} if its relative commutant in $\mathcal B$ is finite dimensional modulo its Jacobson radical (see \Cref{def Q}).

\begin{customthm}{\ref{thm UKP-3}}
Let $\mathcal{B}$ be a unital Banach algebra, $T_1,T_2\in\mathcal{B}$, and $n\in\mathbb{N}$ such that $T_1^{(n)}\sim T_2^{(n)}$ in $\mathbb{M}_n(\mathcal{B})$.
If $T_1$ has essentially finite-dimensional commutant in $\mathcal{B}$, then $T_1\sim T_2$ in $\mathcal{B}$.
\end{customthm}

The paper is organized as follows.
In \Cref{sec pre}, we introduce similarity numbers and similarity functions and establish a reduction for type $\mathrm{I}_n$ von Neumann algebras in \Cref{prop reduction}.
In \Cref{sec bound}, we prove upper and lower bounds for these functions and construct compact counterexamples in \Cref{thm UKP-1}.
\Cref{sec VNA} characterizes the von Neumann algebras with separable predual for which the Ulam--Kaplansky problem holds.
In \Cref{sec Banach}, we prove a cancellation theorem for elements with essentially finite-dimensional commutants in unital Banach algebras and apply it to unilateral weighted shifts.

\section{Preliminaries}
\label{sec pre}

For every operator $A$, we abbreviate $A^{(2)}=A\oplus A$.
For every invertible operator $X$, let $\kappa(X)=\|X\|\|X^{-1}\|$.

\begin{definition}\label{def similarity}
For $A,B\in\mathbb M_n$, their \emph{similarity number} is defined by
\begin{equation*}
	s(A,B)
	=\inf\left\{
    \kappa(X)\colon X\in\mathbb M_n~\text{is invertible},~AX=XB
    \right\},
\end{equation*}
where $\inf\varnothing=\infty$.
For every integer $n\geqslant1$, the \emph{similarity function}
$\theta_n\colon[1,\infty)\to[1,\infty]$ is defined by
\begin{equation*}
	\theta_n(t)
	=\sup\left\{
    s(A,B)\colon
    A,B\in\mathbb M_n,\
    s(A^{(2)},B^{(2)})\leqslant t
    \right\}.
\end{equation*}
\end{definition}

Whenever $s(A,B)<\infty$, the infimum defining $s(A,B)$ is attained.
Indeed, a minimizing sequence $\{X_j\}_{j=1}^\infty$ in $\mathbb M_n$ normalized by $\|X_j\|=1$ has uniformly bounded inverses.
By compactness, a subsequence converges to an invertible matrix $X$ satisfying $AX=XB$ and $\kappa(X)=s(A,B)$.
Moreover, by multiplying an invertible intertwiner $X$ by a positive scalar, we may assume that $\|X\|=\|X^{-1}\|=\sqrt{\kappa(X)}$.

Adjoining a common scalar eigenvalue allows us to compare the similarity functions in consecutive dimensions.

\begin{lemma}\label{lem theta-increasing}
For every $t\geqslant 1$, we have $\theta_n(t)\leqslant\theta_{n+1}(t)$.
\end{lemma}

\begin{proof}
Let $A,B\in\mathbb M_n$, $\lambda>\max\{\|A\|,\|B\|\}$, $A_\lambda=A\oplus\lambda$, and $B_\lambda=B\oplus\lambda$.
Since $\lambda$ lies outside both spectra, every invertible intertwiner between $A_\lambda$ and $B_\lambda$ has the form $Y\oplus\mu$, where $AY=YB$ and $\kappa(Y)\leqslant\kappa(Y\oplus\mu)$.
Conversely, an invertible $Y$ satisfying $AY=YB$ can be rescaled to satisfy $\|Y\|=\|Y^{-1}\|\geqslant1$.
Then $Y\oplus1$ intertwines $A_\lambda$ and $B_\lambda$, with $\kappa(Y\oplus1)=\kappa(Y)$.
Hence
\begin{equation*}
    s(A_\lambda,B_\lambda)=s(A,B).
\end{equation*}
The same argument shows that $s(A_\lambda^{(2)},B_\lambda^{(2)})$.
Taking suprema proves the lemma.
\end{proof}

The following proposition connects uniform bounds for matrix similarity numbers with similarity cancellation in type $\mathrm{I}_n$ von Neumann algebras.

\begin{proposition}\label{prop reduction}
For $n\geqslant 1$, the following statements are equivalent:
\begin{enumerate}
\item[$(1)$] $\theta_n(t)<\infty$ for every $t\geqslant 1$;
\item[$(2)$] the Ulam--Kaplansky problem holds for $\mathbb{M}_n(\ell^\infty)$;
\item[$(3)$] the Ulam--Kaplansky problem holds for every type $\mathrm{I}_n$ von Neumann algebra.
\end{enumerate}
\end{proposition}

\begin{proof}
The implication $(3)\Rightarrow(2)$ is immediate.

$(2)\Rightarrow(1)$: Suppose that $\theta_n(t)=\infty$ for some $t\geqslant1$.
Then for each $j\geqslant 1$, there exists a pair $(A_j,B_j)\in \mathbb{M}_n ^2$ such that
\begin{equation*}
	s(A_j^{(2)},B_j^{(2)})\leqslant t,\qquad
	s(A_j,B_j)\geqslant j.
\end{equation*}
By rescaling, we may assume that $\max\{\|A_j\|,\|B_j\|\}\leqslant1$.
For each $j\geqslant 1$, choose an invertible $X_j\in\mathbb M_{2n}$ satisfying
\begin{equation*}
	A_j^{(2)}X_j=X_jB_j^{(2)},\qquad
	\|X_j\|=\|X_j^{-1}\|\leqslant\sqrt t.
\end{equation*}
We define three operators by
\begin{equation*}
	A=\bigoplus_{j=1}^{\infty}A_j,
	\qquad
	B=\bigoplus_{j=1}^{\infty}B_j,
	\qquad
	X=\bigoplus_{j=1}^{\infty}X_j.
\end{equation*}
Then $A,B\in\mathbb{M}_n(\ell^\infty)$, $X\in\mathbb{M}_{2n}(\ell^\infty)$, and $A^{(2)}X=XB^{(2)}$.
By $(2)$, there exists an invertible $Y=\bigoplus_{j=1}^{\infty}Y_j\in\mathbb{M}_n(\ell^\infty)$ satisfying $AY=YB$.
Consequently,
\begin{equation*}
	j\leqslant s(A_j,B_j)\leqslant\kappa(Y_j)\leqslant\kappa(Y)
	\quad(j\geqslant1),
\end{equation*}
a contradiction.

$(1)\Rightarrow(3)$: Let $\mathcal M$ be a type $\mathrm{I}_n$ von Neumann algebra.
We may assume that $\mathcal M=\mathbb M_n(C(\Omega))$, where $\Omega$ is Stonian.
Suppose that $A,B\in\mathcal M$ and $X\in\mathbb M_2(\mathcal M)$ is invertible with $A^{(2)}X=XB^{(2)}$.
Let $t=\kappa(X)$.
For every $\omega\in\Omega$, evaluation gives
\begin{equation*}
	s(A(\omega)^{(2)},B(\omega)^{(2)})\leqslant t,
\end{equation*}
and hence $s(A(\omega),B(\omega))\leqslant\theta_n(t)$.
Thus, there is an invertible matrix $S(\omega)\in\mathbb M_n$ satisfying
\begin{equation*}
	A(\omega)S(\omega)=S(\omega)B(\omega),\qquad
	\|S(\omega)\|=\|S(\omega)^{-1}\|\leqslant\sqrt{\theta_n(t)}.
\end{equation*}
By \cite[Theorem 1]{DP-64}, there exists an invertible operator $Y\in\mathcal{M}$ such that $A=Y^{-1}BY$, $\|Y\|\leqslant1+\sqrt{\theta_n(t)}$, and $\|Y^{-1}\|\leqslant1+\sqrt{\theta_n(t)}$.
This proves $(3)$.
\end{proof}

\section{Estimates of the similarity functions}
\label{sec bound}

In this section, we provide upper and lower bounds for the similarity functions defined in \Cref{def similarity}, and use the lower bound to prove \Cref{thm UKP-1}.

\subsection{Upper bound}

In this subsection, we prove the upper bound for $\theta_n$ by induction on the matrix size.
The main step is \Cref{lem induction}, which extracts a common summand of dimension at least half the original size.
For the cancellation argument, we use the following elementary inequality obtained from singular values:
\begin{equation}\label{equ det}
	\|X^{-1}\|\leqslant\frac{\|X\|^{d-1}}{|\det X|}\qquad(X\in\mathbb M_d~\text{invertible}).
\end{equation}
We first estimate the similarity number of a pair after cancelling a common direct summand.

\begin{lemma}\label{lem cancel}
Let $A,B\in\mathbb M_n$ and $C\in\mathbb M_m$.
Then
\begin{equation*}
	s(A,B)\leqslant\bigl(3s(A\oplus C,B\oplus C)\bigr)^{m+1}.
\end{equation*}
\end{lemma}

\begin{proof}
We may assume $K:=s(A\oplus C,B\oplus C)<\infty$.
Let $X$ be an invertible operator in $\mathbb M_{n+m}$ such that $\|X\|=K$, $\|X^{-1}\|=1$, and $(A\oplus C)X=X(B\oplus C)$.
Write
\begin{equation*}
	X=\begin{pmatrix}X_{11}&X_{12}\\X_{21}&X_{22}\end{pmatrix}.
\end{equation*}
Let $0<\rho<1$.
Since $\det(X_{22}+zI_m)$ is a monic polynomial of degree $m$, we have
\begin{equation*}
	\frac1{2\pi\mathbf i}\int_{|z|=\rho}z^{-m-1}\det(X_{22}+z I_m)\,dz=1.
\end{equation*}
Hence $|\det(X_{22}+\lambda I_m)|\geqslant\rho^m$ for some $|\lambda|=\rho$.
Let
\begin{equation*}
	Y=\begin{pmatrix}X_{11}&X_{12}\\X_{21}&H\end{pmatrix},\qquad
	H=X_{22}+\lambda I_m.
\end{equation*}
It follows from \eqref{equ det} that $\|H^{-1}\|\leqslant(K+\rho)^{m-1}/\rho^m$.
Moreover, we have
\begin{equation*}
	(A\oplus C)Y=Y(B\oplus C),\qquad \|X^{-1}(Y-X)\|\leqslant\rho<1.
\end{equation*}
Thus, $Y$ is invertible with $\|Y^{-1}\|\leqslant(1-\rho)^{-1}$.
Let $W=X_{11}-X_{12}H^{-1}X_{21}$.
Then
\begin{equation*}
	Y=\begin{pmatrix}I_n&X_{12}H^{-1}\\0&I_m\end{pmatrix}
	\begin{pmatrix}W&0\\0&H\end{pmatrix}
	\begin{pmatrix}I_n&0\\H^{-1}X_{21}&I_m\end{pmatrix}.
\end{equation*}
Consequently, $W$ is invertible and
\begin{equation*}
	Y^{-1}=
	\begin{pmatrix}
		W^{-1}&-W^{-1}X_{12}H^{-1}\\
		-H^{-1}X_{21}W^{-1}&H^{-1}+H^{-1}X_{21}W^{-1}X_{12}H^{-1}
	\end{pmatrix}.
\end{equation*}
Since $(B\oplus C)Y^{-1}=Y^{-1}(A\oplus C)$, we have $BW^{-1}=W^{-1}A$, and hence $AW=WB$.
Moreover, we have
\begin{equation*}
	\|W^{-1}\|\leqslant\|Y^{-1}\|\leqslant\frac1{1-\rho},\qquad
	\|W\|\leqslant K+K^2\frac{(K+\rho)^{m-1}}{\rho^m}
	\leqslant K(1+K/\rho)^m.
\end{equation*}
It follows that
\begin{equation*}
	s(A,B)\leqslant\frac{K}{1-\rho}(1+K/\rho)^m.
\end{equation*}
Therefore, $s(A,B)\leqslant2K(2K+1)^m\leqslant(3K)^{m+1}$ by taking $\rho=1/2$.
\end{proof}

We next record a similarity that turns an idempotent into an orthogonal projection while preserving norms on its range and kernel.
For an idempotent $E$, let
\begin{equation*}
	J_E=2E-I,\qquad D_E=E^*E+(I-E)^*(I-E).
\end{equation*}
It is clear that $D_E=\frac{1}{2}\bigl(I+J_E^*J_E\bigr)\geqslant\frac{1}{2}I$.
Thus, $D_E$ is invertible and
\begin{equation*}
	\|D_E\|=\frac{1}{2}(1+\|J_E\|^2),\qquad \|D_E^{-1}\|=\frac{2\|J_E\|^2}{1+\|J_E\|^2}.
\end{equation*}
Consequently, we have $\kappa(D_E)=\|J_E\|^2$.
Let
\begin{equation}\label{equ S-E}
	S_E=D_E^{1/2},\qquad P_E=S_EES_E^{-1}.
\end{equation}
Since $D_EE=E^*D_E$, it follows that $P_E$ is a projection and
\begin{equation}\label{equ kappa-SE}
	\kappa(S_E)=\|J_E\|=\|2E-I\|.
\end{equation}
The bound in \eqref{equ kappa-SE} is optimal: if $SES^{-1}=P$ is a projection, then it is clear that $\|2E-I\|=\|S^{-1}(2P-I)S\|\leqslant\kappa(S)$.
The identity
\begin{equation}\label{equ isometry}
	\|S_Ex\|^2=\|Ex\|^2+\|(I-E)x\|^2
\end{equation}
shows that $S_E$ acts isometrically on each of $\operatorname{ran} E$ and $\ker E$, sending them onto $\operatorname{ran} P_E$ and $\ker P_E$.
Thus, if $AE=EA$, the similarity $S_EAS_E^{-1}$ is block diagonal relative to the orthogonal decomposition $\operatorname{ran} P_E\oplus\ker P_E$.

We next choose a spectral cut that works simultaneously for matrices with the same characteristic polynomial.

\begin{lemma}\label{lem cut}
Let $(G_j)_{j\in J}$ be a nonempty family of matrices in $\mathbb M_n$ with the same characteristic polynomial and $\|G_j\|\leqslant t$ for every $j\in J$, where $n\geqslant1$ and $t\geqslant1$.
Then for every $0<r<1/2$, there exists $\rho\in[1/2-r,1/2+r]$ such that the circle $|z|=\rho$ is disjoint from their spectra and their inside Riesz idempotents $E_j$ satisfy
\begin{equation*}
	\kappa(S_{E_j})\leqslant(4t/r)^n\qquad (j\in J).
\end{equation*}
Moreover, $G_iX=XG_j$ implies $E_iX=XE_j$ for every $X\in\mathbb M_n$.
\end{lemma}

\begin{proof}
Let $\{\lambda_k\}_{k=1}^n$ be the common eigenvalues of $(G_j)_{j\in J}$, counted with multiplicity.
Let $p(x)=\prod_{k=1}^n(x-|\lambda_k|)$.
Then Chebyshev's extremal theorem \cite[Section~3.3]{MH-03}, rescaled to $[1/2-r,1/2+r]$, gives
\begin{equation*}
	\max_{|x-1/2|\leqslant r}|p(x)|\geqslant2(r/2)^n.
\end{equation*}
Choose $\rho\in[1/2-r,1/2+r]$ such that $|p(\rho)|\geqslant2(r/2)^n$.
For every $j\in J$ and every $|z|=\rho$, we have
\begin{equation*}
	|\det(zI-G_j)|
	=\prod_{k=1}^n|z-\lambda_k|
	\geqslant\prod_{k=1}^n\bigl|\rho-|\lambda_k|\bigr|
	=|p(\rho)|\geqslant2(r/2)^n.
\end{equation*}
Since $\|zI-G_j\|\leqslant|z|+\|G_j\|\leqslant2t$, it follows from \eqref{equ det} that
\begin{equation*}
	\|(zI-G_j)^{-1}\|
	\leqslant\frac{\|zI-G_j\|^{n-1}}{|\det(zI-G_j)|}
	\leqslant\frac{(2t)^{n-1}}{2(r/2)^n}
	=\frac14(4/r)^nt^{n-1}.
\end{equation*}
Let $c_0=\frac14(4/r)^nt^{n-1}$.
For each $j\in J$, the inside Riesz idempotent $E_j$ is given by
\begin{equation*}
	E_j=\frac1{2\pi\mathbf i}\int_{|z|=\rho}(zI-G_j)^{-1}\,dz.
\end{equation*}
Then $\|E_j\|\leqslant\rho c_0\leqslant c_0$.
Since $c_0=\frac{1}{r}(4t/r)^{n-1}\geqslant1$, it follows from \eqref{equ kappa-SE} that
\begin{equation*}
	\kappa(S_{E_j})\leqslant 2c_0+1\leqslant4c_0=(4/r)^nt^{n-1}\leqslant(4t/r)^n.
\end{equation*}
Finally, if $G_iX=XG_j$ for some $X\in\mathbb M_n$, then
\begin{equation*}
	(zI-G_i)^{-1}X=X(zI-G_j)^{-1}.
\end{equation*}
Integrating over $|z|=\rho$, we obtain $E_iX=XE_j$.
\end{proof}

We now combine the common spectral cut with the idempotent similarity above to extract a common summand of dimension at least half the original size, while controlling the condition numbers of the required similarities.

\begin{lemma}\label{lem induction}
Suppose that $A,B\in\mathbb M_n$ and $s(A^{(2)},B^{(2)})\leqslant t$, where $n\geqslant1$ and $t\geqslant1$.
Then there are integers $m\geqslant1$ and $0\leqslant k\leqslant m$ with $n=k+m$, matrices $A_0,B_0\in\mathbb M_k$, $C\in\mathbb M_m$, and invertible matrices $S,T\in\mathbb M_n$ such that
\begin{equation*}
	S^{-1}AS=A_0\oplus C,\qquad T^{-1}BT=B_0\oplus C,\qquad\kappa(S)\kappa(T)\leqslant(10t)^{5m}.
\end{equation*}
\end{lemma}

\begin{proof}
Choose an invertible matrix $X$ such that
\begin{equation*}
	A^{(2)}X=XB^{(2)},\qquad\|X\|\leqslant t,\qquad \|X^{-1}\|=1.
\end{equation*}
Write $X=(X_{ij})_{i,j=1}^2$ and $X^{-1}=(Y_{ij})_{i,j=1}^2$, where $X_{ij},Y_{ij}\in\mathbb M_n$.
Let
\begin{equation*}
	G_1=X_{11}Y_{11},\qquad
	G_2=Y_{11}X_{11},\qquad
	G_3=Y_{22}X_{22}.
\end{equation*}
These matrices commute with $A,B,B$, respectively, and have
norm at most $t$.
By the elementary identity $\det(zI-UV)=\det(zI-VU)$ and
\begin{equation*}
	G_1=I-X_{12}Y_{21},\qquad G_3=I-Y_{21}X_{12},
\end{equation*}
the operators $\{G_j\}_{j=1}^3$ have the same characteristic
polynomial.

Let $0<r<1/2$.
Apply \Cref{lem cut} and let $E_1,E_2,E_3$ be the resulting
inside Riesz idempotents for a common circle $|z|=\rho$,
where $\rho\in[1/2-r,1/2+r]$.
They have a common rank $d$, and \eqref{equ kappa-SE} gives
\begin{equation*}
	\kappa(S_{E_j})=\kappa(S_{I-E_j})\leqslant(4t/r)^n\qquad(1\leqslant j\leqslant3).
\end{equation*}
Define
\begin{equation*}
	(E,F,Z,W)=
	\begin{cases}
      (E_1,E_3,X_{12},Y_{21}),&d\geqslant n/2,\\
      (I-E_1,I-E_2,X_{11},Y_{11}),&d<n/2.
    \end{cases}
\end{equation*}
Let $m=\operatorname{rank}E=\operatorname{rank}F$ and $k=n-m$.
Then $m\geqslant1$ and $0\leqslant k\leqslant m$.
The block identities and the last assertion of \Cref{lem cut} imply that
\begin{equation*}
	AE=EA,\qquad BF=FB,\qquad AZ=ZB,\qquad EZ=ZF,\qquad FW=WE.
\end{equation*}
Let $Z_0\colon\operatorname{ran}F\to\operatorname{ran}E$ and $W_0\colon\operatorname{ran}E\to\operatorname{ran}F$ be the restrictions of $Z$ and $W$.
The matrix $H=Z_0W_0$ has norm at most $t$.
Since $ZW=I-G_1$ in the first case and $ZW=G_1$ in the second, every eigenvalue of $H$ has modulus at least
\begin{equation*}
	\min\{\rho,1-\rho\}\geqslant1/2-r.
\end{equation*}
Thus, $H$ is invertible.
Since both spaces have dimension $m$, $Z_0$ is invertible with $Z_0^{-1}=W_0H^{-1}$.
By $\|Z_0\|\leqslant t$, $\|W_0\|\leqslant1$, and \eqref{equ det}, we obtain
\begin{equation*}
	\kappa(Z_0)
	\leqslant t\|H^{-1}\|
	\leqslant\frac{t^m}{|\det H|}
	\leqslant\left(\frac{t}{1/2-r}\right)^m.
\end{equation*}
Let $P_0=0_k\oplus I_m$.
Then there are unitary matrices $U_A,U_B$ such that
\begin{equation*}
	U_A^*P_EU_A=P_0,\qquad U_B^*P_FU_B=P_0.
\end{equation*}
Let $L_A=U_A^*S_E$ and $L_B=U_B^*S_F$.
Then
\begin{equation*}
	L_AEL_A^{-1}=L_BFL_B^{-1}=P_0,\qquad\kappa(L_A),\kappa(L_B)\leqslant(4t/r)^n.
\end{equation*}
The identities $AE=EA$, $BF=FB$, and $EZ=ZF$ therefore give
\begin{equation*}
	L_AAL_A^{-1}=A_0\oplus C,\qquad
    L_BBL_B^{-1}=B_0\oplus B_1,\qquad
    L_AZL_B^{-1}=R_0\oplus R,
\end{equation*}
where the first and second blocks have dimensions $k$ and $m$, respectively.
Since $AZ=ZB$, we have $CR=RB_1$.

By \eqref{equ isometry}, the restrictions of $L_A$ and $L_B$ to $\operatorname{ran}E$ and $\operatorname{ran}F$ are isometries onto $\{0\}\oplus\mathbb C^m$.
Thus $R$ represents $Z_0$ under these isometric identifications. In particular, $R$ is invertible and
\begin{equation*}
	\kappa(R)=\kappa(Z_0)\leqslant\left(\frac{t}{1/2-r}\right)^m.
\end{equation*}
Rescale $R$ such that $\|R\|=\|R^{-1}\|=\sqrt{\kappa(R)}$.
Then $\kappa(I_k\oplus R^{-1})=\kappa(R)$.
Define
\begin{equation*}
	S=L_A^{-1},\qquad T=L_B^{-1}(I_k\oplus R^{-1}).
\end{equation*}
Since $RB_1R^{-1}=C$, it follows that $S^{-1}AS=A_0\oplus C$ and $T^{-1}BT=B_0\oplus C$.
Moreover, we have
\begin{equation*}
	\kappa(S)\kappa(T)
	\leqslant\kappa(L_A)\kappa(L_B)\kappa(R)
	\leqslant\left(\frac{4t}{r}\right)^{2n}
	\left(\frac{t}{1/2-r}\right)^m.
\end{equation*}
Let $r=2/5$.
Since $n=k+m\leqslant2m$, we obtain
\begin{equation*}
	\kappa(S)\kappa(T)
	\leqslant(10t)^{2n+m}
	\leqslant(10t)^{5m}.
\end{equation*}
This completes the proof.
\end{proof}

We now combine \Cref{lem induction,lem cancel} to prove the upper bound for $\theta_n$.

\begin{proposition}\label{prop upper}
For every integer $n\geqslant1$ and every $t\geqslant 1$, we have
\begin{equation*}
	\theta_n(t)\leqslant(10t)^{24^n}.
\end{equation*}
\end{proposition}

\begin{proof}
We prove the proposition by induction on $n$.
If $n=1$, then $\theta_n(t)=1$ and the conclusion is clear. 
Suppose that $n\geqslant 2$.
Let $t\geqslant1$, $q=10t$, and $A,B\in\mathbb M_n$ such that $s(A^{(2)},B^{(2)})\leqslant t$.
By \Cref{lem induction}, there are invertible $S,T\in\mathbb M_n$
such that
\begin{equation*}
	S^{-1}AS=A_0\oplus C,\qquad
	T^{-1}BT=B_0\oplus C,\qquad
	\kappa(S)\kappa(T)\leqslant q^{5m},
\end{equation*}
where $n=k+m$, $m\geqslant1$, and $0\leqslant k\leqslant m$.
If $k=0$, then
\begin{equation*}
	s(A,B)\leqslant\kappa(ST^{-1})\leqslant q^{5n}\leqslant q^{24^n}.
\end{equation*}
Next we assume that $1\leqslant k<n$.
Then
\begin{equation*}
	s(A_0^{(2)}\oplus C^{(2)},B_0^{(2)}\oplus C^{(2)})\leqslant t\kappa(S)\kappa(T)\leqslant tq^{5m}.
\end{equation*}
Since $3t\leqslant q$, \Cref{lem cancel} yields
\begin{equation*}
	s(A_0^{(2)},B_0^{(2)})
	\leqslant(3tq^{5m})^{2m+1}
	\leqslant q^{p_m},
	\qquad p_m=(5m+1)(2m+1).
\end{equation*}
By induction, we have
\begin{equation*}
	s(A_0,B_0)
	\leqslant\theta_k(q^{p_m})
	\leqslant(10q^{p_m})^{24^k}
	\leqslant q^{(p_m+1)24^k}.
\end{equation*}
Since $5m+p_m+1=2(m+1)(5m+1)\leqslant2\cdot2^m\cdot6^m\leqslant24^m$, it follows that
\begin{equation*}
	s(A,B)\leqslant\kappa(S)\kappa(T)s(A_0,B_0)
	\leqslant q^{5m+(p_m+1)24^k}
    \leqslant q^{(5m+p_m+1)24^k}
    \leqslant q^{24^{m+k}}
    =q^{24^n}.
\end{equation*}
This completes the proof.
\end{proof}

\subsection{Lower bound}

For the lower bound, we construct matrices using evaluation vectors in two Hilbert spaces of functions with even and odd boundary values.
A unitary matrix multiplier in \Cref{lem multiplier} gives a controlled intertwiner between their twofold ampliations.
An intertwiner between the original matrices produces a scalar
function that is odd on the boundary; its forced zero yields
the lower bound.

For $f\in L^1(\mathbb R^2)\cap L^2(\mathbb R^2)$, its \emph{Fourier transform} is $\widehat f(\xi)=\frac{1}{2\pi}\int_{\mathbb R^2}f(x)e^{-\mathbf ix\cdot\xi}\,dx$.
We use the unitary extension of this transform to $L^2(\mathbb R^2)$.
Define the Hilbert space
\begin{equation*}
	\mathcal H=\{f\in L^2(\mathbb R^2)\colon w\widehat f\in L^2(\mathbb R^2)\},
	\qquad
	w(\xi)=(1+\xi_1^2)(1+\xi_2^2),
\end{equation*}
with inner product $\langle f,g\rangle_{\mathcal H}=\int_{\mathbb R^2}\widehat f(\xi)\,\overline{\widehat g(\xi)}\,w(\xi)^2\,d\xi$.
Since $\int_{\mathbb R}(1+s^2)^{-2}\,ds=\pi/2$, we have $\|w^{-1}\|_2=\pi/2$.
For $f\in\mathcal H$, the Cauchy--Schwarz inequality gives
\begin{equation*}
	\|\widehat f\|_1\leqslant\|w^{-1}\|_2\|w\widehat f\|_2
	=\frac{\pi}{2}\|f\|_{\mathcal H}.
\end{equation*}
By the \emph{Fourier inversion formula} $f(x)=\frac{1}{2\pi}\int_{\mathbb R^2}\widehat f(\xi)e^{\mathbf ix\cdot\xi}\,d\xi$ and the Riemann--Lebesgue lemma, we have
\begin{equation*}
	f\in C_0(\mathbb R^2),\qquad \|f\|_\infty\leqslant\frac{1}{4}\|f\|_{\mathcal H}.
\end{equation*}
In particular, point evaluation is bounded on $\mathcal H$.

Fix $R>0$ and let $S_R=[-R,R]^2$.
Define the subspaces
\begin{equation*}
	\mathcal H_\pm
	=\{f\in\mathcal H\colon f(-x)=\pm f(x)~\text{for all}~x\in\partial S_R\}.
\end{equation*}
These subspaces are closed because point evaluations are bounded.
For each $x\in S_R$, the Riesz representation theorem gives a unique vector $e_\pm(x)\in\mathcal H_\pm$ satisfying
\begin{equation*}
	f(x)=\langle f,e_\pm(x)\rangle_{\mathcal H}
	\qquad (f\in\mathcal H_\pm).
\end{equation*}
We first record the bounds and continuity estimates needed for the evaluation vectors.

\begin{lemma}\label{lem evaluation}
For every $x\in S_R$, we have
\begin{equation*}
	\|e_+(x)\|_{\mathcal H}\geqslant\frac{1}{4\sqrt2},
	\qquad
	\|e_\pm(x)\|_{\mathcal H}\leqslant\frac{1}{4},
	\qquad
	\operatorname{Lip}(e_\pm)\leqslant\frac{1}{4}.
\end{equation*}
On the boundary $\partial S_R$, we have $e_\pm(-x)=\pm e_\pm(x)$.
Moreover, for any distinct interior points $\{x_j\}_{j=1}^n$ of $S_R$, the vectors $\{e_\pm(x_j)\}_{j=1}^n$ are linearly independent in $\mathcal H_\pm$ for each choice of sign.
\end{lemma}

\begin{proof}
It is routine to verify that $\|(\eta\cdot\xi)w^{-1}\|_2=\pi/2\|\eta\|$ for every vector $\eta\in\mathbb R^2$.
By $|e^{\mathbf is}-1|\leqslant|s|$, the Cauchy--Schwarz inequality, and the Fourier inversion formula, we have $|f(x)-f(x')|\leqslant\|x-x'\|\|f\|_{\mathcal H}/4$.
By the definition of $e_\pm(x)$, we have
\begin{align*}
	\|e_\pm(x)\|_{\mathcal H}
	&=\sup_{\substack{f\in\mathcal H_\pm\\\|f\|_{\mathcal H}\leqslant1}}
	|f(x)|\leqslant\frac14,\\
	\|e_\pm(x)-e_\pm(x')\|_{\mathcal H}
	&=\sup_{\substack{f\in\mathcal H_\pm\\\|f\|_{\mathcal H}\leqslant1}}
	|f(x)-f(x')|\leqslant\frac14|x-x'|.
\end{align*}
Let $K_x\in\mathcal H$ be the evaluation vector defined by $f(x)=\langle f,K_x\rangle_{\mathcal H}$ for $f\in\mathcal H$.
The Fourier inversion formula shows that
\begin{equation*}
	\widehat K_x(\xi)=\frac{e^{-\mathbf ix\cdot\xi}}{2\pi w(\xi)^2},
	\qquad
	K_x(x)=\frac{1}{(2\pi)^2}\int_{\mathbb R^2}w(\xi)^{-2}\,d\xi
	=\frac{1}{16}.
\end{equation*}
Since the function $y\mapsto e^{-|y_1|-|y_2|}$ has Fourier transform $2/(\pi w(\xi))$, it follows that $K_x\geqslant0$ by convolution and translation.
Since $w$ is even, we have $K_{-x}(y)=K_x(-y)$ for every $y\in\mathbb R^2$.
Thus, $g_x=(K_x+K_{-x})/2$ is even and belongs to $\mathcal H_+$.
By the definition of $K_z$, we have
\begin{equation*}
	\|g_x\|_{\mathcal H}^2=\frac{g_x(x)+g_x(-x)}{2}
	=g_x(x)\geqslant\frac{K_x(x)}{2}=\frac{1}{32}.
\end{equation*}
Since $g_x(x)=\langle g_x,e_+(x)\rangle_{\mathcal H}$, the Cauchy--Schwarz inequality gives
\begin{equation*}
	\|e_+(x)\|_{\mathcal H}\geqslant\|g_x\|_{\mathcal H}
	\geqslant\frac{1}{4\sqrt2}.
\end{equation*}
On $\partial S_R$, the identity $f(-x)=\pm f(x)$ for $f\in\mathcal H_\pm$ gives $e_\pm(-x)=\pm e_\pm(x)$.
For distinct interior points $\{x_j\}_{j=1}^n$, choose smooth functions $\{f_j\}_{j=1}^n$ supported in disjoint compact interior neighborhoods with $f_j(x_j)=1$.
Their Fourier transforms are rapidly decreasing, and their boundary values vanish, hence $f_j\in\mathcal H_+\cap\mathcal H_-$.
The identities $\langle f_i,e_\pm(x_j)\rangle_{\mathcal H}=\delta_{ij}$ give linear independence.
\end{proof}

We next construct a unitary-valued matrix function whose multiplication operator exchanges the two boundary conditions.

\begin{lemma}\label{lem multiplier}
There is a smooth unitary-valued function $U\colon\mathbb R^2\to\mathbb M_2$ satisfying
\begin{enumerate}
\item[$(1)$] $U(-x)=-U(x)$ for $x\in\partial S_R$;
\item[$(2)$] multiplication by $U$ defines a bounded invertible map $T\colon\mathcal H_+\otimes\mathbb C^2\to\mathcal H_-\otimes\mathbb C^2$ with $\|T\|,\|T^{-1}\|\leqslant \exp(\pi/R)$;
\item[$(3)$] $T^*(e_-(x)\otimes v)=e_+(x)\otimes U(x)^*v$ for every $x\in S_R$ and $v\in\mathbb C^2$.
\end{enumerate}
\end{lemma}

\begin{proof}
Let
\begin{equation*}
	P=\begin{pmatrix}1&0\\0&-1\end{pmatrix},\qquad
	Q=\begin{pmatrix}0&1\\1&0\end{pmatrix},\qquad
	\alpha=\frac{\pi}{2R}.
\end{equation*}
Since $PQ=-QP$, for $x=(x_1,x_2)\in\mathbb R^2$ the product
\begin{equation*}
	U(x_1,x_2)=e^{\mathbf{i}\alpha x_1P/2}e^{\mathbf{i}\alpha x_2Q}e^{\mathbf{i}\alpha x_1P/2}
	      =\cos(\alpha x_2)e^{\mathbf{i}\alpha x_1P}+\mathbf{i}\sin(\alpha x_2)Q
\end{equation*}
defines a smooth unitary-valued function with even part
\begin{equation*}
	\frac{1}{2}\bigl(U(x_1,x_2)+U(-x_1,-x_2)\bigr)=\cos(\alpha x_1)\cos(\alpha x_2)I_2.
\end{equation*}
This gives $(1)$.
For $j=1,2$, since $|\partial_{\xi_j}\log w|=2|\xi_j|/(1+\xi_j^2)\leqslant1$, it follows that
\begin{equation*}
	|\log w(\xi+s\mathbf e_j)-\log w(\xi)|\leqslant|s|,\qquad w(\xi+s\mathbf e_j)\leqslant e^{|s|}w(\xi),
\end{equation*}
where $\mathbf e_1=(1,0)$, $\mathbf e_2=(0,1)$, and $s\in\mathbb R$.
Let $M_{s,j}f(x)=e^{\mathbf isx_j}f(x)$.
Then
\begin{equation*}
	\widehat{M_{s,j}f}(\xi)=\widehat f(\xi-s\mathbf e_j),\qquad \|M_{s,j}\|\leqslant e^{|s|}.
\end{equation*}
Since $P$ and $Q$ are self-adjoint with eigenvalues $\{1,-1\}$, diagonalizing each matrix separately expresses its exponential multiplier as a direct sum of two scalar multipliers.
For multiplication by $U$ and $U^*$, the three factors and their inverses give the bound
\begin{equation*}
	e^{\alpha/2}e^\alpha e^{\alpha/2}=e^{2\alpha}=e^{\pi/R}.
\end{equation*}
Both $U$ and $U^*$ are odd on the boundary, hence the multipliers restrict to inverse maps between the indicated subspaces.
This proves $(2)$.
Finally, for every $f\in\mathcal H_+\otimes\mathbb C^2$, $x\in S_R$, and $v\in\mathbb C^2$, we have
\begin{align*}
	\langle f,T^*(e_-(x)\otimes v)\rangle
	&=\langle Tf,e_-(x)\otimes v\rangle
	=\langle U(x)f(x),v\rangle_{\mathbb C^2}\\
	&=\langle f(x),U(x)^*v\rangle_{\mathbb C^2}
	=\langle f,e_+(x)\otimes U(x)^*v\rangle.
\end{align*}
This proves $(3)$.
\end{proof}

We now construct matrices by using evaluation vectors as eigenvectors.
The simple-spectrum condition will also be used in \Cref{prop lower,thm UKP-2}.

\begin{lemma}\label{lem simple-spectrum}
For every integer $n\geqslant1$ and every $1<q\leqslant 2$, there are matrices $A,B\in\mathbb M_n$ with simple spectrum $\{1,2,\ldots,n\}$ such that
\begin{equation*}
	s(A^{(2)},B^{(2)})\leqslant q,\qquad
	s(A,B)\geqslant\frac{\lfloor\sqrt n\rfloor\log q}{16\pi}.
\end{equation*}
\end{lemma}

\begin{proof}
Let $m=\lfloor\sqrt n\rfloor$ and $R=2\pi/\log q$.
Choose $n$ distinct interior points $\{x_j\}_{j=1}^n$ that include the centers of the $m\times m$ subdivision of $S_R$ into equal squares.
Every point of $S_R$ lies within $\sqrt2 R/m$ of one of these centers.

Let $\mathcal E_\pm$ be the spans of $\{e_\pm(x_j)\}_{j=1}^n$ in $\mathcal H_\pm$.
By \Cref{lem evaluation}, these vectors form bases, and we can define operators $A_\pm$ on $\mathcal E_\pm$ by
\begin{equation*}
	A_\pm e_\pm(x_j)=je_\pm(x_j)
	\qquad (1\leqslant j\leqslant n).
\end{equation*}
Orthonormal bases identify these operators with matrices in $\mathbb M_n$.
Both matrices have simple spectrum $\{1,2,\ldots,n\}$.
By \Cref{lem multiplier}, $T^*$ restricts to an invertible map $X\colon\mathcal E_-\otimes\mathbb C^2\to\mathcal E_+\otimes\mathbb C^2$ satisfying
\begin{equation*}
	(A_+\otimes I_2)X=X(A_-\otimes I_2),
	\qquad
	\kappa(X)\leqslant\kappa(T)\leqslant q.
\end{equation*}
It follows that $s(A_+^{(2)},A_-^{(2)})\leqslant q$.
It remains to obtain a lower bound for $s(A_+,A_-)$.
By \Cref{lem multiplier}~$(3)$, we have
\begin{align*}
	\|e_+(x)\|_{\mathcal H}
	&=\|e_+(x)\otimes U(x)^*v\|
	=\|T^*(e_-(x)\otimes v)\|\\
	&\leqslant\|T\|\|e_-(x)\|_{\mathcal H}
	\leqslant\sqrt2\|e_-(x)\|_{\mathcal H},
\end{align*}
where $\|T\|\leqslant e^{\pi/R}=\sqrt q\leqslant\sqrt2$ and $v\in\mathbb C^2$ is a unit vector.
By \Cref{lem evaluation}, it follows that
\begin{equation*}
	\|e_-(x)\|_{\mathcal H}\|e_+(x)\|_{\mathcal H}
	\geqslant\frac{\|e_+(x)\|_{\mathcal H}^2}{\sqrt2}
	\geqslant\frac{1}{32\sqrt2}
	\qquad (x\in S_R).
\end{equation*}
Let $Y\colon\mathcal E_-\to\mathcal E_+$ be any invertible map satisfying $A_+Y=YA_-$.
Let $P_-$ be the orthogonal projection of $\mathcal H_-$ onto $\mathcal E_-$, and define
\begin{equation*}
	F(x)=\langle YP_-e_-(x),e_+(x)\rangle_{\mathcal H}
	\qquad (x\in S_R).
\end{equation*}
By \Cref{lem evaluation}, $F$ is odd on $\partial S_R$ and has Lipschitz constant at most $\|Y\|/8$.
Since each eigenvalue of $A_+$ is simple and $A_+Y=YA_-$, each $Ye_-(x_j)$ is a scalar multiple of $e_+(x_j)$.
Consequently, we have
\begin{equation*}
	|F(x_j)|
	=\|Ye_-(x_j)\|_{\mathcal H}\|e_+(x_j)\|_{\mathcal H}
	\geqslant\frac{\|e_-(x_j)\|_{\mathcal H}\|e_+(x_j)\|_{\mathcal H}}{\|Y^{-1}\|}
	\geqslant\frac{1}{32\sqrt2\|Y^{-1}\|}.
\end{equation*}
Since $F\colon S_R\to\mathbb C\cong\mathbb R^2$ is continuous and odd on $\partial S_R$, it has a zero in $S_R$ by the Borsuk--Ulam theorem.
Choose $x_0\in S_R$ with $F(x_0)=0$.
Then there exists some $x_j$ such that $|x_j-x_0|\leqslant\sqrt2 R/m$, and it follows that
\begin{equation*}
	\frac{1}{32\sqrt2\|Y^{-1}\|}
	\leqslant |F(x_j)-F(x_0)|
	\leqslant\frac{\|Y\|}{8}\frac{\sqrt2 R}{m}.
\end{equation*}
It follows that
\begin{equation*}
	\kappa(Y)\geqslant\frac{m}{8R}
	=\frac{m\log q}{16\pi}.
\end{equation*}
Taking the infimum over all such $Y$ gives $s(A_+,A_-)\geqslant\frac{m\log q}{16\pi}$.
The matrices $A=A_+$ and $B=A_-$ have the required properties.
\end{proof}

The following scaling construction converts the preceding estimate for $1<q\leqslant2$ into a bound proportional to $t-1$ for every $t>1$.

\begin{lemma}\label{lem scaling}
Let $A,B\in\mathbb M_n$ be similar matrices, and let $a\geqslant1$.
Define
\begin{equation*}
	\widetilde A=\begin{pmatrix} A&aI_n\\0&A\end{pmatrix},
	\qquad
	\widetilde B=\begin{pmatrix} B&I_n\\0&B\end{pmatrix}.
\end{equation*}
If $A$ and $B$ are similar to diagonal matrices, then
\begin{equation*}
	s(\widetilde A,\widetilde B)=a\,s(A,B),\qquad
	s(\widetilde A^{(2)},\widetilde B^{(2)})=a\,s(A^{(2)},B^{(2)}).
\end{equation*}
\end{lemma}

\begin{proof}
Let $W$ be invertible in $\mathbb M_{2n}$ such that $\widetilde A W=W\widetilde B$, and let $N(x,y)=(y,0)$ on $\mathbb C^n\oplus\mathbb C^n$.
If $v$ belongs to a generalized $\lambda$-eigenspace of $\widetilde B$, then $Wv$ belongs to the corresponding one of $\widetilde A$.
Since $A$ and $B$ are similar to diagonal matrices, we have
\begin{equation*}
	aNWv
	=(\widetilde A-\lambda I_{2n})Wv
	=W(\widetilde B-\lambda I_{2n})v
	=WNv.
\end{equation*}
These generalized eigenspaces span $\mathbb C^{2n}$, and hence $aNW=WN$.
It follows that
\begin{equation*}
	W=
	\begin{pmatrix}
		aY& Z\\
		0& Y
	\end{pmatrix},
\end{equation*}
where $Y$ is invertible in $\mathbb M_n$.
The identity $\widetilde A W=W\widetilde B$ gives $AY=YB$.
The diagonal blocks of $W$ and $W^{-1}$ give $\kappa(W)\geqslant a\kappa(Y)$.
Conversely, each invertible $Y$ with $AY=YB$ gives the intertwiner
$W=\operatorname{diag}(aY,Y)$ satisfying $\kappa(W)=a\kappa(Y)$.
Taking infima proves the first identity.
The second identity follows by applying the first to the operators $A^{(2)},B^{(2)}$.
\end{proof}

We now combine \Cref{lem simple-spectrum,lem scaling} with the monotonicity of $\theta_n$ to obtain the lower bound.

\begin{proposition}\label{prop lower}
For every integer $n\geqslant2$ and every $t>1$, we have
\begin{equation*}
	\theta_n(t)\geqslant\frac{t-1}{500}\sqrt n .
\end{equation*}
\end{proposition}

\begin{proof}
Let $p=\lfloor n/2\rfloor$, $m=\lfloor\sqrt p\rfloor$,
$a=(t+1)/2$, and $q=t/a\in(1,2)$.
It is clear that
\begin{equation*}
    m=\lfloor\sqrt{n/2}\rfloor
    \geqslant\frac12\sqrt{\frac n2}
    >\frac{\sqrt n}{3},
\end{equation*}
and
\begin{equation*}
    a\log q\geqslant a(q-1)\log2
    =\frac{t-1}{2}\log2>\frac{t-1}{3}.
\end{equation*}
It follows from \Cref{lem theta-increasing,lem simple-spectrum,lem scaling} that
\begin{equation*}
    \theta_n(t)\geqslant\theta_{2p}(t)
    \geqslant\frac{am\log q}{16\pi}
    >\frac{(t-1)\sqrt n}{144\pi}
    >\frac{t-1}{500}\sqrt n.
\end{equation*}
This completes the proof.
\end{proof}

\subsection{The Ulam--Kaplansky problem in \texorpdfstring{$\mathcal B(\mathcal H)$}{B(H)}}

As an application of \Cref{prop lower}, we obtain the following theorem.

\begin{theorem}\label{thm UKP-1}
Let $\mathcal H$ be an infinite-dimensional Hilbert space.
For every $t>1$, there exist block-diagonal compact operators $A,B\in\mathcal B(\mathcal H)$ and an invertible operator $X\in\mathcal B(\mathcal H\oplus\mathcal H)$ such that
\begin{equation*}
    A^{(2)}X=XB^{(2)},\qquad
    \|X\|\|X^{-1}\|\leqslant t,
    \qquad
    A\nsim B.
\end{equation*}
\end{theorem}

\begin{proof}
Let $t>1$.
By \Cref{prop lower}, for each $j\geqslant 1$, we can choose invertible matrices $A_j,B_j\in\mathbb M_{n_j}$ satisfying
\begin{equation*}
	s(A_j^{(2)},B_j^{(2)})\leqslant t,
	\qquad
	s(A_j,B_j)>j.
\end{equation*}
By translation and scaling, we may assume that
\begin{equation*}
	\max\{\|A_j\|,\|B_j\|\}\leqslant 2^{-j}\quad(j\geqslant 1),\qquad \sigma(A_i)\cap\sigma(A_j)=\varnothing\quad(i\ne j).
\end{equation*}
Write $\mathcal H=\mathcal H_0\oplus\bigoplus_{j=1}^{\infty}\mathbb C^{n_j}$ and define
\begin{equation*}
	A=0_{\mathcal H_0}\oplus\bigoplus_{j=1}^{\infty}A_j,
	\qquad
	B=0_{\mathcal H_0}\oplus\bigoplus_{j=1}^{\infty}B_j.
\end{equation*}
These operators are compact.
Taking direct sums of intertwiners for the doubled pairs, balanced to have norms and inverse norms at most $\sqrt t$, and adjoining the identity on $\mathcal H_0\oplus\mathcal H_0$, gives an invertible operator $X$ satisfying
\begin{equation*}
    A^{(2)}X=XB^{(2)},\qquad
    \max\{\|X\|,\|X^{-1}\|\}\leqslant\sqrt t.
\end{equation*}
Suppose that $AY=YB$ for some invertible $Y\in\mathcal B(\mathcal H)$.
Then $Y=Y_0\oplus\bigoplus_{j=1}^{\infty}Y_j$ by Rosenblum's theorem \cite{Rosenblum-56}.
It follows that $A_jY_j=Y_jB_j$, and hence
\begin{equation*}
	j<s(A_j,B_j)\leqslant\kappa(Y_j)\leqslant\kappa(Y)
	\qquad(j\geqslant1),
\end{equation*}
a contradiction.
\end{proof}

\section{The Ulam--Kaplansky problem in von Neumann algebras}
\label{sec VNA}

For a von Neumann algebra $\mathcal M$, the Ulam--Kaplansky problem asks whether $A^{(2)}\sim B^{(2)}$ in $\mathbb M_2(\mathcal M)$ implies $A\sim B$ in $\mathcal M$ for all $A,B\in\mathcal M$.
Recall that $\mathcal M$ is \emph{subhomogeneous} if there are integers $n_1<n_2<\cdots<n_m$ and abelian von Neumann algebras $\{\mathcal A_k\}_{k=1}^m$ such that
\begin{equation*}
    \mathcal M\cong\bigoplus_{k=1}^m\mathbb M_{n_k}(\mathcal A_k).
\end{equation*}
To apply the matrix construction inside corners of a von Neumann algebra, we need to retain the lower bound for similarity numbers.
The following lemma achieves this by forcing intertwiners to have coefficients in a maximal abelian von Neumann subalgebra.

\begin{lemma}\label{lem MASA}
Let $\mathcal N$ be a von Neumann algebra with separable predual and let $H$ be a self-adjoint element in $\mathcal N$ such that $0\leqslant H\leqslant I_{\mathcal N}/4$ and $W^*(H)$ is maximal abelian in $\mathcal N$.
Suppose that $A,B\in\mathbb M_n$ have simple spectrum $\{1,2,\ldots,n\}$ and 
\begin{equation*}
    U=A\otimes I_{\mathcal N}+I_n\otimes H,
    \qquad
    V=B\otimes I_{\mathcal N}+I_n\otimes H.
\end{equation*}
If $X\in\mathbb M_n(\mathcal N)$ is invertible and $UX=XV$, then there exists an invertible matrix $Y\in\mathbb M_n$ such that $AY=YB$ and $\kappa(Y)\leqslant\kappa(X)$.
\end{lemma}

\begin{proof}
Let $\mathcal A=W^*(H)$.
Since $\mathcal A$ is maximal abelian in $\mathcal N$, we have $\{H\}'\cap\mathcal N=\mathcal A$.

For $1\leqslant k\leqslant n$, let $D_k=\{z\in\mathbb C\colon |z-k|<1/3\}$.
Define a holomorphic function $f$ on $\bigcup_{k=1}^nD_k$ by $f(z)=z-k$ for $z\in D_k$.
Since $A$ and $B$ are similar to diagonal matrices, we have
\begin{equation*}
    \sigma(U)\cup\sigma(V)
    \subseteq\bigcup_{k=1}^n[k,k+1/4],
    \qquad
    f(U)=f(V)=I_n\otimes H.
\end{equation*}
By $UX=XV$ and holomorphic functional calculus, we have $(I_n\otimes H)X=X(I_n\otimes H)$.
Consequently, $X,X^{-1}\in\mathbb M_n(\mathcal A)$, and the relation $UX=XV$ reduces to
\begin{equation*}
    (A\otimes I_{\mathcal N})X
    =X(B\otimes I_{\mathcal N}).
\end{equation*}
Choose a character $\chi\colon\mathcal A\to\mathbb C$ and let $\pi=\operatorname{id}_{\mathbb M_n}\otimes\chi$.
Since $\pi$ is a unital $*$-homomorphism, $Y=\pi(X)$ is invertible with $Y^{-1}=\pi(X^{-1})$,
and $AY=YB$.
Moreover,
\begin{equation*}
    \kappa(Y)
    =\|\pi(X)\|\|\pi(X^{-1})\|
    \leqslant\|X\|\|X^{-1}\|
    =\kappa(X).
\end{equation*}
This completes the proof.
\end{proof}

We now apply \Cref{lem MASA} to the matrix pairs from \Cref{lem simple-spectrum} to obtain the following characterization.

\begin{theorem}\label{thm UKP-2}
Let $\mathcal M$ be a von Neumann algebra with separable predual.
Then the Ulam--Kaplansky problem holds for $\mathcal M$ if and only if $\mathcal M$ is subhomogeneous.
\end{theorem}

\begin{proof}
If $\mathcal M$ is subhomogeneous, the conclusion follows from \Cref{prop reduction,prop upper} applied to its finitely many homogeneous central summands.

Conversely, suppose that $\mathcal M$ is not subhomogeneous and let $t>1$.
By \Cref{lem simple-spectrum} and positive scalar rescaling of the intertwiners, we can choose positive integers $\{n_j\}_{j=1}^\infty$, matrices $A_j,B_j\in\mathbb M_{n_j}$ with simple spectrum $\{1,2,\ldots,n_j\}$, and invertible $X_j\in\mathbb M_{2n_j}$ such that
\begin{equation*}
    A_j^{(2)}X_j=X_jB_j^{(2)},\qquad
    \|X_j\|,\|X_j^{-1}\|\leqslant t,\qquad
    s(A_j,B_j)>j.
\end{equation*}
By the type decomposition theorem \cite[Chapter~6]{KR2-86}, we can choose a sequence $\{P_j\}_{j=1}^\infty$ of mutually orthogonal nonzero projections in $\mathcal M$ such that
\begin{equation*}
    P_j\mathcal MP_j=\mathbb M_{n_j}(\mathcal N_j).
\end{equation*}
Since $\mathcal N_j$ has separable predual, there exists $0\leqslant H_j\leqslant I_{\mathcal N_j}/4$ such that $W^*(H_j)$ is maximal abelian in $\mathcal N_j$.
Let
\begin{equation*}
    U_j=\varepsilon_j
        (A_j\otimes I_{\mathcal N_j}+I_{n_j}\otimes H_j),
    \qquad
    V_j=\varepsilon_j
        (B_j\otimes I_{\mathcal N_j}+I_{n_j}\otimes H_j),
\end{equation*}
where $\varepsilon_j>0$ is sufficiently small such that $\max\{\|U_j\|,\|V_j\|\}\leqslant1/4$.

Let $P_0=I_{\mathcal M}-\sum_{j\geqslant1}P_j$ and $U_0=V_0=0$.
Define two operators in $\mathcal M$ by
\begin{equation*}
    A=\sum_{j=0}^{\infty}2^{-j}(P_j+U_j),
    \qquad
    B=\sum_{j=0}^{\infty}2^{-j}(P_j+V_j).
\end{equation*}
Let $X=P_0^{(2)}+\sum_{j=1}^{\infty}(X_j\otimes I_{\mathcal N_j})\in\mathbb M_2(\mathcal M)$.
The sums defining $P_0$ and $X$ converge strongly, whereas those defining $A$ and $B$ converge in norm.
The uniform bounds on $X_j$ and $X_j^{-1}$ imply that $X$ is invertible, and $A^{(2)}X=XB^{(2)}$.

It remains to show that $A\nsim B$.
Suppose that $AY=YB$ for some invertible $Y\in\mathcal M$.
Let $Y_{ij}=P_iYP_j$ for $i,j\geqslant0$.
Then $(2^{-i}-2^{-j})Y_{ij}=2^{-j}Y_{ij}V_j-2^{-i}U_iY_{ij}$, and hence
\begin{equation*}
    |2^{-i}-2^{-j}|\|Y_{ij}\|
    \leqslant\frac{2^{-i}+2^{-j}}{4}\|Y_{ij}\|.
\end{equation*}
For $i\ne j$, we have $|2^{-i}-2^{-j}|\geqslant(2^{-i}+2^{-j})/3$, and hence $Y_{ij}=0$.
Thus, $Y$ commutes with every $P_j$.
Consequently, $Y_{jj}$ is invertible in $P_j\mathcal MP_j$, with inverse $P_jY^{-1}P_j$, and $U_jY_{jj}=Y_{jj}V_j$.
By applying \Cref{lem MASA}, we obtain
\begin{equation*}
    j<s(A_j,B_j)\leqslant\kappa(Y_{jj})\leqslant\kappa(Y)
    \qquad(j\geqslant1),
\end{equation*}
which is impossible.
This completes the proof.
\end{proof}

\section{The Ulam--Kaplansky problem in Banach algebras}
\label{sec Banach}

In this section, we study the Ulam--Kaplansky problem in unital Banach algebras.
For a unital Banach algebra $\mathcal{B}$, its \emph{(Jacobson) radical} $\mathrm{Rad}(\mathcal{B})$ is the intersection of all maximal left (right) ideals of $\mathcal{B}$ \cite[Theorem 5.9]{Allan-11}.
It is a closed two-sided ideal, and $\mathcal B/\mathrm{Rad}(\mathcal B)$ is semisimple by \cite[Theorem 5.13]{Allan-11}.
Moreover, we have
\begin{equation*}
    \begin{split}
        \mathrm{Rad}(\mathcal{B})
        &=\{A\in\mathcal{B}\colon\sigma(AB)=\sigma(BA)=\{0\}~\text{for all}~B\in\mathcal{B}\}\\
        &=\{A\in\mathcal{B}\colon\sigma(B_1AB_2)=\{0\}~\text{for all}~B_1,B_2\in\mathcal{B}\}.
    \end{split}
\end{equation*}
The reader is referred to \cite[Section 5.3]{Allan-11} for more details.

\begin{definition}\label{def Q}
Let $\mathcal B$ be a unital Banach algebra.
For every $T\in\mathcal B$, its \emph{relative commutant} in $\mathcal B$ and the \emph{semisimplification} of this commutant are defined by
\begin{equation*}
	(T,\mathcal B)'=\{X\in\mathcal B\colon TX=XT\},\qquad
	\mathcal Q(T,\mathcal B)
	=(T,\mathcal B)'/\mathrm{Rad}((T,\mathcal B)').
\end{equation*}
We say that $T$ has \emph{essentially finite-dimensional commutant} in $\mathcal B$ if $\mathcal Q(T,\mathcal B)$ is finite dimensional.
\end{definition}

Two elements $A,B\in\mathcal B$ are \emph{similar}, denoted by $A\sim B$, if $AX=XB$ for some invertible $X\in\mathcal B$.
We first show that invertibility and similarity of idempotents lift from the quotient by the Jacobson radical.

\begin{lemma}\label{lem radical-lifting}
Let $\mathcal B$ be a unital Banach algebra, $n\geqslant1$, and $\mathcal J=\mathrm{Rad}(\mathcal B)$.
Let
\begin{equation*}
	\pi\colon\mathbb M_n(\mathcal B)\to\mathbb M_n(\mathcal B/\mathcal J)
\end{equation*}
be the entrywise quotient map.
If $X\in\mathbb M_n(\mathcal B)$ with $\pi(X)$ invertible, then $X$ is also invertible.
Moreover, if $P,Q\in\mathbb M_n(\mathcal B)$ are idempotents and $\pi(P)\sim\pi(Q)$, then $P\sim Q$ in $\mathbb M_n(\mathcal B)$.
\end{lemma}

\begin{proof}
Since $\mathrm{Rad}(\mathbb M_n(\mathcal B))=\mathbb M_n(\mathrm{Rad}(\mathcal B))$ by \cite[Section~4, pp.~57--58]{Lam-01}, we have
\begin{equation*}
	\ker\pi=\mathrm{Rad}(\mathbb M_n(\mathcal B)).
\end{equation*}
If $\pi(X)$ is invertible, then $\pi(X)^{-1}=\pi(Y)$ for some element $Y\in\mathbb M_n(\mathcal B)$.
Thus, both $XY,YX\in I_n+\ker\pi$ are invertible.
Therefore, $X$ is invertible.

Suppose that $\pi(P)\pi(X)=\pi(X)\pi(Q)$ for some invertible $\pi(X)$.
Let
\begin{equation*}
	Y=PXQ+(I-P)X(I-Q).
\end{equation*}
Then $PY=YQ$ and $\pi(Y)=\pi(X)$.
Since $Y$ is invertible by the first assertion, we obtain $P\sim Q$.
\end{proof}

For every $T\in\mathcal B$ and integer $n\geqslant1$, let $T^{(n)}=\operatorname{diag}(T,\ldots,T)\in\mathbb M_n(\mathcal B)$.
Combining \Cref{lem radical-lifting} with a rank comparison in a finite-dimensional semisimple quotient gives the following cancellation theorem.

\begin{theorem}\label{thm UKP-3}
Let $\mathcal B$ be a unital Banach algebra, $T_1,T_2\in\mathcal B$, and $n\in\mathbb N$ such that $T_1^{(n)}\sim T_2^{(n)}$ in $\mathbb M_n(\mathcal B)$.
If $T_1$ has essentially finite-dimensional commutant in $\mathcal B$, then $T_1\sim T_2$ in $\mathcal B$.
\end{theorem}

\begin{proof}
Let $\mathcal C=(T_1,\mathcal B)'$ and $\mathcal J=\mathrm{Rad}(\mathcal C)$.
Choose an invertible $X\in\mathbb M_n(\mathcal B)$ satisfying $T_1^{(n)}X=XT_2^{(n)}$.
Let $\{e_{ij}\}_{1\leqslant i,j\leqslant n}$ be the canonical system of matrix units in $\mathbb M_n$, and define
\begin{equation*}
	E_{ij}=e_{ij}\otimes I_{\mathcal B},\qquad F_{ij}=XE_{ij}X^{-1}.
\end{equation*}
Then both families are systems of matrix units in $\mathbb M_n(\mathcal C)$.
By the hypothesis and Wedderburn's theorem \cite[Theorem 5.22]{Allan-11}, we have
\begin{equation*}
    \mathbb M_n(\mathcal C/\mathcal J)
    \cong\bigoplus_{k=1}^m\mathbb M_{nd_k}
\end{equation*}
for some positive integers $\{d_k\}_{k=1}^m$.
Let $\pi\colon\mathbb M_n(\mathcal C)\to \mathbb M_n(\mathcal C/\mathcal J)$ be the entrywise quotient map.
In each summand, the images of $\{F_{jj}\}_{j=1}^n$ have equal ranks and sum to the identity.
Their common rank is therefore $d_k$, which is also the rank of the image of $E_{11}$.
Thus, $\pi(E_{11})\sim\pi(F_{11})$.
By \Cref{lem radical-lifting}, there is an invertible
$W\in\mathbb M_n(\mathcal C)$ satisfying $E_{11}W=WF_{11}$.
It follows that
\begin{equation*}
	E_{11}WX=WXE_{11},\qquad
	T_1^{(n)}WX=WXT_2^{(n)}.
\end{equation*}
Since $WX$ is invertible and commutes with $E_{11}$, its compression $Y=(WX)_{11}$ is invertible in $\mathcal B$.
The second identity gives $T_1Y=YT_2$.
\end{proof}

For the weighted shifts in the following corollary, the semisimplification of the commutant is isomorphic to $\mathbb C$, allowing us to apply \Cref{thm UKP-3}.

\begin{corollary}\label{cor weighted-shift}
Let $A\in\mathcal B(\mathcal H)$ be an injective unilateral weighted shift with weights $\{w_j\}_{j=0}^{\infty}$, and let $B\in\mathcal B(\mathcal H)$.
Suppose that $\{|w_j|\}_{j=0}^{\infty}$ is decreasing and belongs to $\ell^p$ for some $0<p<\infty$.
If $A^{(n)}\sim B^{(n)}$ for some integer $n\geqslant1$, then $A\sim B$.
\end{corollary}

\begin{proof}
By \cite[p.~99, Corollary~2]{Shields-74}, $A$ is strictly cyclic.
Moreover, $A$ is quasinilpotent as $w_j\to0$.
It follows from \cite[p.~94, Proposition~31 and Corollary~1]{Shields-74} that the commutative Banach algebra $\{A\}'=(A,\mathcal B(\mathcal H))'$ has a unique maximal ideal.
Consequently, we have $\mathcal Q(A,\mathcal B(\mathcal H))\cong\mathbb C$.
The conclusion follows from \Cref{thm UKP-3}.
\end{proof}

\section*{Acknowledgments}

After the current paper was typed up, we became aware of Aouichaoui’s work \cite{Aoui-26}, whose Theorem 5.3 characterizes Kaplansky's second test problem in finite type I von Neumann algebras.


\begin{thebibliography}{99}


\bibitem{Allan-11}
G. Allan.
\emph{Introduction to Banach spaces and algebras.}
Prepared for publication and with a preface by H. Garth Dales.
Oxf. Grad. Texts Math., 20.
Oxford University Press, Oxford, 2011. viii+371 pp.
MR2761146

\bibitem{Aoui-26}
M. Aouichaoui.
\emph{Similarity cancellation and primitive operator roots.}
arXiv:2611.01298v1 [math.FA], 2026

\bibitem{Davidson-96}
K. Davidson.
\emph{$C^*$-algebras by example.}
Fields Inst. Monogr., 6
American Mathematical Society, Providence, RI, 1996. xiv+309 pp.
MR1402012


\bibitem{DH-90}
K. Davidson, D. Herrero.
\emph{The Jordan form of a bitriangular operator.}
J. Funct. Anal. 94 (1990), no. 1, 27--73.
MR1077544


\bibitem{DP-64}
D. Deckard, C. Pearcy.
\emph{On continuous matrix-valued functions on a Stonian space.}
Pacific J. Math. 14 (1964), 857--869.
MR0172130


\bibitem{Jiang-04}
C. Jiang.
\emph{Similarity classification of Cowen--Douglas operators.}
Canad. J. Math. 56 (2004), no. 4, 742--775.
MR2074045


\bibitem{KR2-86}
R. Kadison, J. Ringrose.
\emph{Fundamentals of the theory of operator algebras, II, Advanced theory.}
Academic Press, Orlando, FL, 1986.
MR0859186


\bibitem{KS-57}
R. Kadison, I. Singer.
\emph{Three test problems in operator theory.}
Pacific J. Math. 7 (1957), 1101--1106.
MR0092123


\bibitem{Kaplansky-54}
I. Kaplansky.
\emph{Infinite abelian groups.}
University of Michigan Press, Ann Arbor, MI, 1954. v+91 pp.
MR0065561


\bibitem{Lam-01}
T. Lam.
\emph{A first course in noncommutative rings. Second edition.}
Grad. Texts in Math., 131.
Springer-Verlag, New York, 2001. xx+385 pp.
MR1838439


\bibitem{MRTZ-24}
L. Marcoux, H. Radjavi, S. Troscheit, Y. Zhang.
\emph{Stability relations for Hilbert space operators and a problem of Kaplansky.}
Math. Z. 308 (2024), no. 3, Paper No. 39, 46 pp.
MR4803020


\bibitem{MH-03}
J. Mason, D. Handscomb.
\emph{Chebyshev Polynomials.}
Chapman \& Hall/CRC, Boca Raton, FL, 2003. xiv+341 pp.
MR1937591


\bibitem{Mauldin-15}
R.~Mauldin (ed.).
\emph{The Scottish Book. Mathematics from the Scottish Caf\'e with selected problems from the new Scottish Book. Second edition.}
Including selected papers presented at the Scottish Book Conference held at North Texas University, Denton, TX, May 1979.
Birkhäuser/Springer, Cham, 2015. xvii+322 pp.
MR3242261


\bibitem{Rosenblum-56}
M. Rosenblum.
\emph{On the operator equation $BX-XA=Q$.}
Duke Math. J. 23 (1956), 263--269.
MR0079235


\bibitem{Shields-74} 
A.~Shields.
\emph{Weighted shift operators and analytic function theory.}
Topics in operator theory, pp. 49--128.
Math. Surveys, No. 13.
American Mathematical Society, Providence, RI, 1974.
MR0361899



\end{thebibliography}
\end{document}